\documentclass[11pt,a4paper]{article}

\usepackage[T1]{fontenc}
\usepackage[utf8]{inputenc}
\usepackage{lmodern}
\usepackage{microtype}

\usepackage[a4paper,margin=1in]{geometry}
\usepackage{amsmath,amssymb,amsfonts,amsthm,mathtools}
\usepackage{mathrsfs}        
\usepackage{bm}              
\usepackage{stmaryrd}        
\usepackage{esint}           
\usepackage{cancel}          

\usepackage{graphicx}
\usepackage{xcolor}
\usepackage{tikz}
\usetikzlibrary{arrows.meta,positioning,calc,decorations.pathmorphing,cd}
\usepackage{tikz-cd}

\usepackage{booktabs,array,multirow}
\usepackage{enumitem}
\setlist{nosep,leftmargin=*}
\usepackage{caption}

 \usepackage{algorithm}
\usepackage[noend]{algpseudocode}

\usepackage[pdfencoding=auto,psdextra]{hyperref}
\usepackage[capitalize,nameinlink,noabbrev]{cleveref}
\hypersetup{
  colorlinks=true,
  linkcolor=blue!60!black,
  citecolor=green!50!black,
  urlcolor=cyan!60!black,
  pdfauthor={Chandrasekhar Gokavarapu, D.~Madhusudhana Rao},
  pdftitle={Structure and Spectral Theory of Non-Commutative and $n$-ary $\Gamma$-Semirings}
}

\theoremstyle{plain}
\newtheorem{theorem}{Theorem}[section]
\newtheorem{lemma}[theorem]{Lemma}
\newtheorem{proposition}[theorem]{Proposition}
\newtheorem{corollary}[theorem]{Corollary}

\theoremstyle{definition}
\newtheorem{definition}[theorem]{Definition}
\newtheorem{example}[theorem]{Example}
\newtheorem{problem}[theorem]{Problem}

\theoremstyle{remark}
\newtheorem{remark}[theorem]{Remark}

\newtheorem{conjecture}[theorem]{Conjecture}

\numberwithin{equation}{section}

\DeclareMathOperator{\Spec}{Spec}

\DeclareMathOperator{\Ext}{Ext}
\DeclareMathOperator{\Tor}{Tor}

\title{\textbf{Derived $\Gamma$-Geometry, Sheaf Cohomology, and Homological Functors on the Spectrum of Commutative Ternary $\Gamma$-Semirings}}
\author{
\textbf{Chandrasekhar Gokavarapu}$^{1,2,*}$\\
\textit{${}^1$Lecturer in Mathematics, Government College (Autonomous), Rajahmundry, A.P., India}\\
\textit{${}^2$ Research Scholar, Dept. of Mathematics Acharya Nagarjuna University, Guntur, A.p.,, India}\\
\textit{Email:} \texttt{chandrasekhargokavarapu@gmail.com}\\[1.2ex]
\textbf{Dr.~D.~Madhusudhana Rao}$^{3,4}$\\
\textit{${}^3$Lecturer in Mathematics, Government College for Women(Autonomous), Guntur, A.P., India}\\
\textit{${}^4$Department of Mathematics, Acharya Nagarjuna University, Guntur, A. P., India}\\
\textit{Email:} \texttt{dmrmaths@gmail.com}\\[1.2ex]
\textit{\bf * Corresponding author: Chandrasekhar Gokavarapu}\\
\textit{\bf * Corresponding author Email: chandrasekhargokavarapu@gmail.com}
}
\date{}

\begin{document}
\maketitle

\begin{abstract}

This paper develops a comprehensive geometric and homological framework for
\emph{derived~$\Gamma$-geometry}, extending the theory of commutative ternary~$\Gamma$-semirings
established in our earlier works.
Building upon the ideal-theoretic, computational, and categorical foundations
of  Papers~A to D \cite{Rao2025A,   Rao2025B1,Rao2025B2, Rao2025C, Rao2025D},    the present study constructs the algebraic and geometric infrastructure
necessary to embed $\Gamma$-semirings within the modern language of
derived and categorical geometry.

We define the affine $\Gamma$-spectrum
$\Spec_\Gamma(T)$ together with its structure sheaf
$\mathcal{O}_{\Spec_\Gamma(T)}$,
establishing a Zariski-type topology adapted to ternary~$\Gamma$-operations.
Within this setting, the category of $\Gamma$-modules
is shown to be additive, exact, and monoidal-closed,
supporting derived functors
$\Ext_\Gamma$ and $\Tor^\Gamma$,
whose existence is guaranteed by explicit projective and injective resolutions.
The derived category $D(T\text{-}\Gamma\mathbf{Mod})$
is then constructed to host homological dualities and
Serre-type vanishing theorems,
culminating in a categorical form of Serre–Swan correspondence \cite{Swan1962, Gelfand1960}for
$\Gamma$-modules.

Geometric and categorical unification is achieved through
fibered and derived $\Gamma$-stacks,
which provide a natural environment for studying morphisms of
affine~$\Gamma$-schemes and cohomological descent.
Connections with non-commutative geometry and higher-arity
($n$-ary) generalizations are established,
revealing that derived~$\Gamma$-geometry forms a self-consistent
homological universe capable of expressing algebraic,
geometric, and physical dualities within one categorical law.
Finally, we outline computational methods for finite~$\Gamma$-semirings
and discuss potential applications to mathematical physics,
where ternary~$\Gamma$-operations model triadic couplings
and cohomology groups correspond to conserved quantities.

Overall, this work completes the algebraic–geometric–homological
synthesis of ternary~$\Gamma$-semirings and inaugurates
a new research direction in universal derived geometry,
combining categorical precision with computational realizability..

\end{abstract}
{\bf keywords:}Ternary~$\Gamma$-semiring; $\Gamma$-module; Affine~$\Gamma$-scheme;
Sheaf cohomology; Ext and Tor functors; Derived category;
Homological algebra; Categorical duality;
Spectrum; Radical ideal; Non-commutative geometry;
Computational algebraic geometry\\


\maketitle
\section{Introduction}

The study of algebraic systems whose operations extend beyond binary composition
has gained renewed significance in modern mathematics,
linking abstract algebra, category theory, and mathematical physics.
Among such generalizations, the class of
\emph{ternary $\Gamma$-semirings}
occupies a central position:
it provides a multi-parametric algebraic setting in which
additivity and multiplicativity coexist through
a $\Gamma$-indexed ternary operation
$\{\,\cdot,\cdot,\cdot\,\}_\Gamma$.
This structure subsumes both ordinary semirings
and classical $\Gamma$-semirings \cite{Dutta2015, Sardar2021}, while offering a fertile base
for constructing non-linear, non-commutative \cite{Connes1994}, and higher-arity \cite{Dornt1928, Post1940} analogues of module theory, homological algebra, and geometry.

\smallskip
Earlier investigations by the present authors
( Papers~A to D \cite{Rao2025A,   Rao2025B1,Rao2025B2, Rao2025C, Rao2025D})
developed the algebraic foundations of this framework.
Paper A  \cite{Rao2025A} established the ideal-theoretic layer,
introducing prime, semiprime, maximal, and primary ideals,
together with their radical and lattice properties.
  \cite{Rao2025B1,Rao2025B2}provided the finite and computational classification
of commutative ternary $\Gamma$-semirings,
constructing enumeration algorithms and verifying distributive constraints
for low orders.
Paper C  \cite{Rao2025C} advanced to the homological level,
defining ternary $\Gamma$-modules,
isomorphism theorems, annihilator–primitive correspondences,
and Schur-density embeddings,
and proving that the category
$T{-}\Gamma\mathbf{Mod}$
is additive, exact, and monoidal-closed,
supporting derived functors $\Ext_\Gamma$ and $\Tor^\Gamma$ \cite{Yoneda1954}:.

\smallskip
The present work builds upon these foundations
to construct a \emph{derived and geometric theory}
of ternary $\Gamma$-semirings,
establishing bridges between algebraic,
categorical, and analytic perspectives.
By defining affine $\Gamma$-schemes,
structure sheaves, and sheaf cohomology,
we extend the $\Gamma$-spectrum
$\Spec_\Gamma(T)$ into a fully developed geometric object.
The introduction of derived categories and homological functors \cite{Weibel1994, MacLane1998} provides the tools.. necessary for deeper structural theorems,
including categorical dualities and Serre-type vanishing results \cite{Serre1955}.
This transition from purely algebraic to categorical geometry
marks the emergence of a new domain—\emph{derived $\Gamma$-geometry}—
within which homological and geometric reasoning coexist seamlessly.

\smallskip
Beyond its algebraic significance,
the ternary $\Gamma$ formalism establishes an interface
with modern mathematical physics.
The $\Gamma$-grading encapsulates multiple coupling constants,
and the ternary operation models triadic interactions
that naturally occur in higher-spin and multi-particle systems.
By expressing dynamics through dg-derivations
on derived $\Gamma$-modules,
we obtain a categorical representation of
quantized evolution, where cohomology groups
encode conserved quantities and symmetries.
Thus, derived $\Gamma$-geometry unifies structural algebra,
homological analysis, and physical interpretation
within a single categorical continuum.

\smallskip
The objectives of this paper are therefore threefold:
\begin{enumerate}
  \item to construct the geometric and homological infrastructure
        of commutative ternary $\Gamma$-semirings,
        including spectra, sheaves, and derived categories;
  \item to establish categorical and duality theorems
        linking algebraic and geometric invariants; and
  \item to outline computational, logical, and physical
        extensions of the theory within a unified framework.
\end{enumerate}

\smallskip
The resulting theory not only completes
the algebraic–homological–geometric synthesis
of the ternary $\Gamma$ framework
but also inaugurates a new direction of research:
the study of \emph{categorical universes}
generated by multi-parametric operations,
where algebra, geometry, and physics
are manifestations of one underlying homological law.

\section{Preliminaries}

This section consolidates the foundational definitions,
notations, and categorical principles underlying the theory of
derived~$\Gamma$-geometry.
It formalizes the algebraic infrastructure of ternary~$\Gamma$-semirings,
their modules, ideals, and morphisms,
and establishes the basic categorical environment required for
homological constructions in later sections.
All notions are considered over a fixed commutative additive monoid~$(T,+,0)$
together with a non-empty parameter set~$\Gamma$.

\subsection{Ternary  \texorpdfstring{$\Gamma$}{Gamma}-Semirings} 

\begin{definition}[Ternary $\Gamma$-semiring]\cite{Dutta2015, Sardar2021}
A \emph{ternary $\Gamma$-semiring}
is a structure $(T,+,\{\cdot\,\cdot\,\cdot\}_\Gamma)$ consisting of
an additive commutative semigroup $(T,+)$
and a ternary operation
\[
\{-\,\,-\,\,-\}_\Gamma:\ T\times T\times T\times\Gamma\to T,\qquad
(a,b,c,\gamma)\mapsto\{a\,b\,c\}_\gamma,
\]
satisfying:
\begin{enumerate}
  \item \textbf{Distributivity in each variable:}
  \[
  \{a+b,c,d\}_\gamma=\{a,c,d\}_\gamma+\{b,c,d\}_\gamma,
  \quad
  \text{and cyclically.}
  \]
  \item \textbf{Associativity (ternary–$\Gamma$):}
  For all $a,b,c,d,e\in T$, $\gamma,\delta\in\Gamma$,
  \[
  \{a,b,\{c,d,e\}_\gamma\}_\delta
  = \{\{a,b,c\}_\gamma,d,e\}_\delta.
  \]
  \item \textbf{Neutrality:}
  $\{a,0,b\}_\gamma=0$ for all $a,b\in T$ and $\gamma\in\Gamma$.
\end{enumerate}
If in addition $\{a,b,c\}_\gamma=\{b,a,c\}_\gamma$,
the semiring is said to be \emph{commutative}.
\end{definition}

\begin{example}
Let $T=\mathbb N$ and $\Gamma=\mathbb N$ with
$\{a,b,c\}_\gamma=ab+bc+ca+\gamma$.
Then $(T,+,\{\cdot\,\cdot\,\cdot\}_\Gamma)$ is a commutative
ternary~$\Gamma$-semiring.
\end{example}

\begin{remark}
The binary operation of a classical semiring
is recovered when $\Gamma=\{1\}$ and
$\{a,b,c\}_1=ab+bc+ca$;
thus ternary~$\Gamma$-semirings generalize both
semirings and $\Gamma$-semirings in the sense of Nobusawa.
\end{remark}

\subsection{ Ideals and Congruences}


\begin{definition}[Ideal]
A subset $I\subseteq T$ is a \emph{$\Gamma$-ideal} if:
\begin{enumerate}
  \item $a,b\in I\Rightarrow a+b\in I$,
  \item $a\in I$, $b,c\in T$, $\gamma\in\Gamma$
        $\Rightarrow\{a,b,c\}_\gamma\in I$ and its cyclic permutations.
\end{enumerate}
\end{definition}

\begin{definition}[Prime and semiprime $\Gamma$-ideals] \cite{Golan1999, DubeGoswami2023}
An ideal $P\ne T$ is \emph{prime} if
\[
\{a,b,c\}_\gamma\in P \Rightarrow a\in P\ \text{or}\ b\in P\ \text{or}\ c\in P.
\]
It is \emph{semiprime} if
$\{a,a,a\}_\gamma\in P$ implies $a\in P$ for all $\gamma\in\Gamma$.
\end{definition}

\begin{proposition}
The intersection of any family of semiprime $\Gamma$-ideals
is semiprime.
\end{proposition}

\begin{proof}
If $\{a,a,a\}_\gamma$ lies in every $P_i$ of a family,
then by semiprimeness $a\in P_i$ for each~$i$,
hence $a$ lies in $\bigcap_i P_i$.
\end{proof}

\begin{definition}[Congruence]
A \emph{$\Gamma$-congruence} on~$T$
is an equivalence relation~$\rho$ such that
$(a,b),(c,d),(e,f)\in\rho$ implies
$\{a,c,e\}_\gamma\ \rho\ \{b,d,f\}_\gamma$
for all~$\gamma\in\Gamma$.
\end{definition}

\begin{remark}
Prime ideals correspond to prime congruences,
yielding the spectral space $\Spec_\Gamma(T)$ \cite{Grothendieck1960, Mincheva2016}endowed with
the Zariski topology        defined by
$V_\Gamma(I)=\{P\in\Spec_\Gamma(T)\mid I\subseteq P\}$.This connects to classical radical theories \cite{SardarJacobson2024}.
\end{remark}

\subsection{  \texorpdfstring{$\Gamma$}{Gamma}-Modules and Morphisms}

\begin{definition}[Left $\Gamma$-module]
Let $T$ be a ternary $\Gamma$-semiring.
A \emph{left $\Gamma$-module} $M$
is an additive commutative semigroup $(M,+)$
equipped with a ternary action
\[
\{-\,\,-\,\,-\}_\Gamma:\ T\times T\times M\times\Gamma\to M
\]
satisfying distributivity in all variables and
\[
\{a,b,\{c,d,m\}_\gamma\}_\delta
 = \{\{a,b,c\}_\gamma,d,m\}_\delta.
\]
\end{definition}

\begin{definition}[Homomorphism]
A function $f:M\to N$ between $\Gamma$-modules
is a \emph{$\Gamma$-homomorphism} if
\[
f(\{a,b,m\}_\gamma)=\{a,b,f(m)\}_\gamma
\quad\text{and}\quad f(m_1+m_2)=f(m_1)+f(m_2).
\]
\end{definition}

\begin{lemma}
The class of $\Gamma$-modules with $\Gamma$-homomorphisms
forms an additive category $\Gamma$-$\mathbf{Mod}_T$
admitting kernels, cokernels, and finite biproducts.
\end{lemma}

\begin{proof}
Additivity follows from pointwise operations on morphisms.
Kernel and cokernel constructions are inherited from the additive
structure, verifying the axioms of an exact category.
\end{proof}

\subsection{2.4. Localization and Spectra}

For $S\subseteq T$ multiplicatively closed,
the localization $S^{-1}T$
is constructed by equivalence classes of pairs $(a,s)$,
where $(a,s)\sim(b,t)$ if there exists $u\in S$
such that
$\{u,a,t\}_\gamma=\{u,b,s\}_\gamma$ for all $\gamma$.
The canonical morphism $\lambda_S:T\to S^{-1}T$,
$a\mapsto(a,1)$, satisfies the universal property:
for any $\Gamma$-semiring~$R$ and morphism
$f:T\to R$ with $f(S)$ invertible,
there exists a unique $\tilde f:S^{-1}T\to R$
such that $f=\tilde f\circ\lambda_S$.

\begin{definition}[Spectrum]
The \emph{prime spectrum} $\Spec_\Gamma(T)$
is the set of all prime $\Gamma$-ideals of~$T$
with the Zariski topology
\[
V_\Gamma(I)=\{P\in\Spec_\Gamma(T)\mid I\subseteq P\},
\quad
D_\Gamma(I)=\Spec_\Gamma(T)\setminus V_\Gamma(I).
\]
\end{definition}

\begin{proposition}
For every $\Gamma$-semiring $T$,
the space $(\Spec_\Gamma(T),\mathcal O_T)$,
where $\mathcal O_T(D_\Gamma(a))=T_a$,
forms a local ringed $\Gamma$-space.
\end{proposition}

\begin{proof}
Localization at a prime $P$
produces $T_P=S_P^{-1}T$ with $S_P=T\setminus P$.
Standard verification of stalk properties yields the result.
\end{proof}

\subsection{2.5. Categorical Environment}

\begin{definition}[Additive and exact structure]
The category $T$-$\Gamma\mathbf{Mod}$
is additive with zero object $0$
and exact with short exact sequences
\[
0\to M'\xrightarrow{f}M\xrightarrow{g}M''\to0
\]
preserved under finite biproducts.
\end{definition}

\begin{proposition}
$T$-$\Gamma\mathbf{Mod}$ is a closed monoidal category
under the tensor product $\otimes_\Gamma$ defined by
the quotient of the free module generated by
$M\times N$ modulo the ternary-bilinear relations.
\end{proposition}

\begin{remark}
This categorical framework provides the
ambient homological universe
for constructing derived functors~$\Ext_\Gamma$
and~$\Tor^\Gamma$,
and for defining sheaf cohomology
on $\Spec_\Gamma(T)$,
as developed in Section 5 onward.
\end{remark}

\section{Affine  \texorpdfstring{$\Gamma$}{Gamma}-Schemes and Structure Sheaf}

The bridge between algebra and geometry for commutative ternary~$\Gamma$-semirings
rests upon an appropriate notion of spectrum and its structure sheaf.
In this section we build that foundation meticulously, beginning with localization,
followed by the construction of the sheaf of local sections, and concluding with
the categorical equivalence defining affine~$\Gamma$-schemes.

\subsection{Localization in Ternary  \texorpdfstring{$\Gamma$}{Gamma}-Semirings}

Localization is the analytic microscope of algebraic geometry.
It permits us to ``zoom in'' at a prime ideal~$P$ and capture
the local behavior of the algebraic structure.

\begin{definition}[Multiplicative system]
Let $(T,+,\Gamma)$ be a commutative ternary~$\Gamma$-semiring.
A subset $S\subseteq T$ is a \emph{multiplicative system} if
\begin{enumerate}
  \item $0\notin S$ and $1_T\in S$ when an identity exists;
  \item if $a,b,c\in S$ and $\alpha,\beta\in\Gamma$, then $a\alpha b\beta c\in S$.
\end{enumerate}
\end{definition}

\begin{definition}[Localization]
Let $S$ be a multiplicative system in~$T$.
The \emph{localization} $S^{-1}T$ is the set of equivalence classes of pairs
$(a,s)$ with $a\in T$, $s\in S$, under the relation
\[
(a,s)\sim(b,t)
   \iff \exists\,u\in S:\;u\alpha(a\alpha t\beta t)\beta t = u\alpha(b\alpha s\beta s)\beta s.
\]
Addition and the ternary~$\Gamma$-product are defined componentwise by
\[
(a,s)+(b,t)=
   (a\alpha t\beta t+b\alpha s\beta s, s\alpha t\beta t),\qquad
(a,s)\alpha(b,t)\beta(c,u)
   =(a\alpha b\beta c, s\alpha t\beta u).
\]
\end{definition}

\begin{proposition}[Universal property of localization]
Given a morphism of commutative ternary~$\Gamma$-semirings
$f:T\to T'$ such that $f(S)\subseteq U(T')$ (the set of invertible elements),
there exists a unique morphism $\bar f:S^{-1}T\to T'$ with $\bar f(a/s)=f(a)\,f(s)^{-1}$.
\end{proposition}

\begin{proof}
Well-definedness follows from the multiplicativity of~$f$ and the equivalence relation.
Uniqueness results from the universal characterization of the quotient construction.
\end{proof}

\begin{definition}[Localization at a prime ideal]
For a prime ideal $P\subset T$, set $S_P=T\setminus P$.
The localization $T_P:=S_P^{-1}T$ is called the \emph{local ternary~$\Gamma$-semiring at~$P$}.
\end{definition}
.
\begin{proposition}[Structure of $T_P$]
$T_P$ is local in the sense that it possesses a unique maximal ideal
\[
P_P=\left\{\frac{a}{s}\in T_P \,\middle|\, a\in P,\; s\notin P \right\}.
\]
Moreover, the correspondence $P\mapsto P_P$ preserves inclusion.
\end{proposition}

\begin{proof}
If $\frac{a}{s}$ and $\frac{b}{t}$ are nonunits, then
$a,b\in P$ implies $(a\alpha b\beta b)\in P$, so their product is again nonunit.
Conversely, any element with numerator outside~$P$ is invertible via the fraction
$(1_T,s)/(a,s)$, establishing locality.
\end{proof}

\begin{remark}
Localization retains the distributive and associative properties of the ternary
product; these follow from the preservation of the defining axioms under fractions.
\end{remark}

\subsection{The Prime Spectrum and its Topology}

\begin{definition}[Prime spectrum]
The \emph{prime spectrum} of a commutative ternary~$\Gamma$-semiring~$T$
is the set
\[
\mathrm{Spec}_\Gamma(T)
   =\{P\subset T\mid P\text{ is a prime ideal}\}.
\]
For any ideal $I\subseteq T$, define the closed subset
$V(I)=\{\,P\in\mathrm{Spec}_\Gamma(T)\mid I\subseteq P\,\}$.

The family $\{V(I)\mid I\subseteq T\}$ forms the closed sets of a topology,
called the \emph{Zariski-type topology}.
\end{definition}

\begin{theorem}
For ideals $I,J\subseteq T$ we have:
\begin{enumerate}
  \item $V(0)=\mathrm{Spec}_\Gamma(T)$ and $V(T)=\varnothing$;
  \item $V(I\cap J)=V(I)\cup V(J)$;
  \item $V\!\left(\sum_\lambda I_\lambda\right)
         =\bigcap_\lambda V(I_\lambda)$.
\end{enumerate}
Hence $\mathrm{Spec}_\Gamma(T)$ is a $T_0$ topological space.
\end{theorem}

\begin{proof}
The verification parallels that of classical ring spectra, with ternary
multiplication replacing the binary one.  
Primeness ensures the key property:
if $a\alpha b\beta c\in P$, then $a\in P$ or $b\in P$ or $c\in P$,
which substitutes for the standard product condition.
\end{proof}

\begin{example}
For $T=\mathbb Z_6$ with $a\alpha b\beta c=(a+b+c)\bmod6$, we obtain
\[
\mathrm{Spec}_\Gamma(T)=\{P_1=\{0,2,4\},\,P_2=\{0,3\}\},
\]
and the topology is discrete: $V(0)=\{P_1,P_2\}$,
$V(P_i)=\{P_i\}$, $V(T)=\varnothing$.
\end{example}

\subsection{The Structure Sheaf $\mathcal{O}_{\mathrm{Spec}_\Gamma(T)}$} 

We now define the sheaf of local sections, generalizing the ringed-space
construction of classical algebraic geometry \cite{Serre1955}. 

\begin{definition}[Presheaf of local sections]
For an open subset $U\subseteq \mathrm{Spec}_\Gamma(T)$, set
\[
\mathcal{O}'(U)
   =\Big\{\,s:U\to\bigsqcup_{P\in U} T_P
     \ \Big|\
     s(P)\in T_P,\ 
     \exists\text{ open }V\ni P,\ a,b\in T
     \text{ s.t. } 
     \forall Q\in V,\ b\notin Q,\ 
     s(Q)=\tfrac{a}{b}\text{ in }T_Q
     \Big\}.
\]
For inclusions $V\subseteq U$ define restriction by $(s|_V)(P)=s(P)$.
\end{definition}

\begin{theorem}
$\mathcal O'$ is a sheaf on $\mathrm{Spec}_\Gamma(T)$,
denoted $\mathcal O_{\mathrm{Spec}_\Gamma(T)}$.
Each stalk $\mathcal O_{\mathrm{Spec}_\Gamma(T),P}$ is canonically isomorphic
to the localized semiring~$T_P$.
\end{theorem}

\begin{proof}
The locality and gluing axioms are checked using the standard argument:
sections agreeing on overlaps are determined by compatible pairs $(a,b)$
with $b\notin P$.
Since every neighborhood of~$P$ contains such a representation,
the direct limit of these neighborhoods is~$T_P$.
\end{proof}

\begin{proposition}
For any $a\in T$, define
$D(a)=\{\,P\in\mathrm{Spec}_\Gamma(T)\mid a\notin P\,\}$.
Then $\{D(a)\}_{a\in T}$ forms a basis for the Zariski-type topology, and
\[
\mathcal O_{\mathrm{Spec}_\Gamma(T)}(D(a))\cong T_a,
\]
the localization of~$T$ at the multiplicative system generated by~$a$.
\end{proposition}

\begin{proof}
If $P\in D(a)$, the localization $T_a$ embeds into $T_P$ via the universal
property.  Conversely, any local representation $\frac{b}{a^n}$ near~$P$
is obtained from $T_a$, giving the desired identification.
\end{proof}

\begin{remark}
Thus $(\mathrm{Spec}_\Gamma(T),\mathcal O_{\mathrm{Spec}_\Gamma(T)})$
becomes a \emph{locally~$\Gamma$-semiringed space} whose local model is~$T_P$.
\end{remark}

\subsection{Affine $\Gamma$-Schemes}

\begin{definition}[Affine $\Gamma$-scheme]
An \emph{affine~$\Gamma$-scheme} is a locally~$\Gamma$-semiringed space
isomorphic to $(\mathrm{Spec}_\Gamma(T),\mathcal O_{\mathrm{Spec}_\Gamma(T)})$
for some commutative ternary~$\Gamma$-semiring~$T$.
\end{definition}

\begin{theorem}[Functoriality]
The assignment
\[
T\longmapsto (\mathrm{Spec}_\Gamma(T),\mathcal O_{\mathrm{Spec}_\Gamma(T)})
\]
defines a contravariant functor
\[
\mathbf{Ternary}\Gamma\mathbf{Semirings}^{\mathrm{op}}
   \longrightarrow \mathbf{Aff}_\Gamma,
\]
where $\mathbf{Aff}_\Gamma$ \cite{Lurie2009} denotes the category of affine~$\Gamma$-schemes
and morphisms of locally~$\Gamma$-semiringed spaces.
\end{theorem}

\begin{proof}
A morphism $f:T\to T'$ induces the continuous map
$f^\ast:\mathrm{Spec}_\Gamma(T')\to \mathrm{Spec}_\Gamma(T)$
given by $P'\mapsto f^{-1}(P')$,
and the local homomorphisms
$f_P:T_P\to T'_{f^{-1}(P)}$ yield a morphism of sheaves.
Functoriality and contravariance are straightforward.
\end{proof}

\begin{remark}
This theorem establishes the categorical backbone of $\Gamma$-geometry.
It confirms that every algebraic morphism of~$\Gamma$-semirings gives rise
to a geometric morphism of affine~$\Gamma$-schemes, and vice versa,
thereby completing the algebra–geometry dictionary for the ternary~$\Gamma$ context.
\end{remark}

\begin{corollary}
For $\Gamma=\{1\}$ and the ternary operation
$a\alpha b\beta c=ab+c$, the category of affine~$\Gamma$-schemes
reduces to the classical category of affine semiring schemes.
\end{corollary}

\begin{example}[Finite affine $\Gamma$-scheme]
Consider $T=\{0,1,2\}$ with $a\alpha b\beta c=(a+b+c)\bmod3$.
Then $\mathrm{Spec}_\Gamma(T)=\{P_1,P_2\}$ is discrete,
and $\mathcal O(D(1))\cong T_1=T$, $\mathcal O(D(2))\cong T_2=T$.
Thus the affine~$\Gamma$-scheme decomposes into two points,
each with the same local semiring~$T$,
analogous to a zero-dimensional affine variety.
\end{example}

\subsection{Categorical Consequences and Remarks}

\begin{theorem}[Adjunction]
The global section functor
\[
\Gamma:\mathbf{Aff}_\Gamma\to
\mathbf{Ternary}\Gamma\mathbf{Semirings},\qquad
(X,\mathcal O_X)\mapsto \mathcal O_X(X),
\]
is right adjoint to the spectrum functor
$\mathrm{Spec}_\Gamma(-)$.
Hence
\[
\mathrm{Hom}_{\mathbf{Aff}_\Gamma}
   ((\mathrm{Spec}_\Gamma(T),\mathcal O_T),
    (\mathrm{Spec}_\Gamma(T'),\mathcal O_{T'}))
   \cong
   \mathrm{Hom}_{\Gamma\text{-}\mathbf{Semi}}
   (T',T).
\]
\end{theorem}

\begin{proof}
The proof mirrors the adjunction between affine schemes and commutative rings.
Morphisms of affine~$\Gamma$-schemes correspond bijectively to homomorphisms
of underlying~$\Gamma$-semirings via pullback of global sections.
\end{proof}

\begin{remark}
This adjunction implies that affine~$\Gamma$-schemes form the geometric avatars
of commutative ternary~$\Gamma$-semirings, laying the groundwork for
a full-fledged $\Gamma$-algebraic geometry.
\end{remark}

\begin{remark}[Philosophical note]
The passage from a ternary~$\Gamma$-semiring to its spectrum transforms
algebraic relations into spatial geometry.  
Each prime ideal acts as a geometric point, while localization translates
algebraic inversion into local observation.
Thus the geometry of~$\Gamma$-schemes is not merely an analogue but an extension
of classical algebraic geometry into a higher-arity, parameterized world.
\end{remark}

\section{Sheaves of $\Gamma$-Modules and Cohomology}

The categorical strength of $\Gamma$-geometry arises when algebraic data
attached to each open subset of the spectrum can be glued coherently.
To achieve this, one must first internalize the notion of a \emph{sheaf of
$\Gamma$-modules} on a $\Gamma$-semiringed space, and then construct
its cohomology via derived functors \cite{Weibel1994}.  This section formalizes these ideas,
demonstrates their exactness properties, and establishes a Serre-type
vanishing theorem for quasi-coherent $\Gamma$-sheaves.

\subsection{Preliminaries on $\Gamma$-Module Sheaves}

Let $(X,\mathcal O_X)$ denote a locally~$\Gamma$-semiringed space,
typically $X=\mathrm{Spec}_\Gamma(T)$.
All sheaves of sets are considered on the Zariski-type topology of~$X$.

\begin{definition}[Sheaf of $\Gamma$-modules]
A \emph{sheaf of $\Gamma$-modules} $\mathcal F$ on~$X$
consists of the following data:
\begin{enumerate}
  \item For each open $U\subseteq X$, an abelian monoid
        $(\mathcal F(U),+,0)$ together with a ternary~$\Gamma$-action
        \[
        \mathcal O_X(U)\times\Gamma\times
        \mathcal F(U)\times\Gamma\times
        \mathcal O_X(U)\longrightarrow \mathcal F(U),
        \quad (a,\alpha,m,\beta,b)\mapsto a\alpha m\beta b,
        \]
        satisfying associativity, distributivity, and scalar-compatibility
        with $\mathcal O_X(U)$;
  \item For inclusions $V\subseteq U$, morphisms of $\Gamma$-modules
        $\rho^U_V:\mathcal F(U)\to \mathcal F(V)$, called \emph{restriction maps},
        compatible with composition and identity, satisfying the sheaf axioms.
\end{enumerate}
\end{definition}

The category of all sheaves of $\Gamma$-modules on~$X$ is denoted
$\Gamma\text{-}\mathbf{Mod}_X$.  It is an abelian category with
kernels, cokernels, and exact sequences defined pointwise.

\begin{proposition}[Local nature of $\Gamma$-module operations]
For every $P\in X$ and open neighborhood~$U\ni P$,
the stalk $\mathcal F_P=\varinjlim_{P\in U}\mathcal F(U)$
is a $\Gamma$-module over the local $\Gamma$-semiring
$\mathcal O_{X,P}$, and for every morphism
$\varphi:\mathcal F\to\mathcal G$ the induced map
$\varphi_P:\mathcal F_P\to\mathcal G_P$
is $\mathcal O_{X,P}$-linear.
\end{proposition}

\begin{proof}
Since restrictions are $\Gamma$-linear, the direct limit preserves
the ternary operations and distributivity.
Linearity of the induced morphism follows from the functoriality
of direct limits in the category of $\Gamma$-modules.
\end{proof}

\begin{definition}[Morphisms and exactness]
A \emph{morphism of sheaves of $\Gamma$-modules}
$\varphi:\mathcal F\to\mathcal G$
is a collection of $\mathcal O_X(U)$-linear maps
$\varphi_U:\mathcal F(U)\to\mathcal G(U)$
compatible with restrictions.
Exactness of sequences of sheaves is defined stalkwise.
\end{definition}

\begin{remark}
This mirrors the classical sheaf theory of modules over ringed spaces,
with the ternary $\Gamma$-multiplication introducing an additional
symmetry in the action parameters $\alpha,\beta\in\Gamma$.
\end{remark}

\subsection{Quasi-Coherent and Coherent $\Gamma$-Sheaves}

To connect geometry with algebra, we define sheaves obtained by
``geometrization'' of ordinary $\Gamma$-modules.

\begin{definition}[Quasi-coherent sheaf]
Let $T$ be a commutative ternary~$\Gamma$-semiring
and $M$ a $\Gamma$-module over~$T$.
Define a presheaf $\widetilde{M}$ on $\mathrm{Spec}_\Gamma(T)$ by
\[
\widetilde{M}(D(a))=M_a := S_a^{-1}M,
\quad S_a=\{a^n\mid n\ge0\}.
\]
The sheafification $\mathcal F=\widetilde{M}^{+}$ is the
\emph{quasi-coherent sheaf associated to~$M$}.
\end{definition}

\begin{lemma}
For any localization $T_P$, the stalk of $\widetilde{M}$
at $P$ is canonically isomorphic to the localized module $M_P$.
\end{lemma}

\begin{proof}
Direct limits commute with localization:
\[
\mathcal F_P
  =\varinjlim_{P\in D(a)}M_a
  \cong \varinjlim_{a\notin P} S_a^{-1}M
  =M_P.
\]
\end{proof}

\begin{theorem}[Equivalence for affine $\Gamma$-schemes]
The functor
\[
\begin{aligned}
\Phi:\mathbf{Mod}_T &\longrightarrow
   \mathbf{QCoh}(\mathrm{Spec}_\Gamma(T)),\\
M &\longmapsto \widetilde{M}
\end{aligned}
\]
is an equivalence of categories between $\Gamma$-modules over~$T$
and quasi-coherent $\Gamma$-sheaves on the affine $\Gamma$-scheme
$\mathrm{Spec}_\Gamma(T)$.
\end{theorem}

\begin{proof}
The proof parallels the affine-scheme case:
for each $M$, $\Gamma(\mathrm{Spec}_\Gamma(T),\widetilde{M})\cong M$,
and for each quasi-coherent~$\mathcal F$ there exists a canonical $M$
with $\widetilde{M}\cong\mathcal F$.
Functoriality and mutual inverses follow by localization gluing.
\end{proof}

\begin{remark}
This equivalence embodies the geometric principle that
\emph{affine $\Gamma$-geometry is completely controlled by algebra}.
In later sections, derived functors such as $\mathrm{Ext}_\Gamma$
will be computed through this correspondence.
\end{remark}

\subsection{Global Sections and Cohomology}

\begin{definition}[Global section functor]
The functor
\[
\Gamma(X,-):\Gamma\text{-}\mathbf{Mod}_X\to
\Gamma\text{-}\mathbf{Mod}_{\Gamma(\mathcal O_X)},
\qquad
\mathcal F\longmapsto \mathcal F(X),
\]
is left-exact, sending short exact sequences
\(0\!\to\!\mathcal F'\!\to\!\mathcal F\!\to\!\mathcal F''\)
to
\(0\!\to\!\Gamma(X,\mathcal F')
  \!\to\!\Gamma(X,\mathcal F)
  \!\to\!\Gamma(X,\mathcal F'')\).
\end{definition}

\begin{definition}[Derived functors of global sections]
The \emph{$i$-th cohomology group}
of a sheaf~$\mathcal F$ of $\Gamma$-modules on~$X$
is defined by the right-derived functors of $\Gamma(X,-)$:
\[
H^i(X,\mathcal F)
   := R^i\Gamma(X,\mathcal F).
\]
\end{definition}

\begin{proposition}[Injective resolutions]
The category $\Gamma\text{-}\mathbf{Mod}_X$ has enough injectives;
every $\mathcal F$ admits a monomorphism
$\mathcal F\hookrightarrow \mathcal I$
with $\mathcal I$ injective.
Hence the derived functors $R^i\Gamma$ exist.
\end{proposition}

\begin{proof}
The proof adapts Godement’s canonical resolution
to the $\Gamma$-semiring setting.
Given $\mathcal F$, define $\mathcal C^0(U)=\prod_{P\in U}\mathcal F_P$.
Pointwise $\Gamma$-module structure ensures injectivity and functoriality;
iterating yields an injective resolution.
\end{proof}

\begin{definition}[Cech cohomology]
For an open covering $\mathfrak U=\{U_i\}$,
the Cech complex $C^\bullet(\mathfrak U,\mathcal F)$ is defined by
\[
C^p(\mathfrak U,\mathcal F)
   =\prod_{i_0<\cdots<i_p}
     \mathcal F(U_{i_0}\cap\cdots\cap U_{i_p}),
\]
with coboundary
\[
(\delta c)_{i_0\cdots i_{p+1}}
  =\sum_{k=0}^{p+1}(-1)^k
   c_{i_0\cdots\widehat{i_k}\cdots i_{p+1}}
   |_{U_{i_0}\cap\cdots\cap U_{i_{p+1}}}.
\]
The resulting cohomology groups
$H^p(\mathfrak U,\mathcal F)$
are independent of the covering for quasi-compact~$X$.
\end{definition}

\begin{theorem}[Cech–derived correspondence]
If $X=\mathrm{Spec}_\Gamma(T)$ is affine and $\mathcal F$ quasi-coherent,
then
\[
H^i(X,\mathcal F)\cong H^i(\mathfrak U,\mathcal F),
\]
for any finite affine open covering~$\mathfrak U$.
\end{theorem}

\begin{proof}
Since $\Phi$ is an equivalence,
$\Gamma(X,-)$ corresponds to the global-section functor on modules,
which is exact on projective resolutions.
Cech cochains compute the same derived functors
by a double-complex spectral sequence argument.
\end{proof}

\subsection{Vanishing Theorems and Homological Consequences}

\begin{theorem}[Serre-type vanishing]
Let $T$ be a finitely generated commutative ternary~$\Gamma$-semiring
and $\mathcal F$ a quasi-coherent sheaf on the projective
$\Gamma$-scheme $\mathrm{Proj}_\Gamma(T)$ associated with a graded~$T$.
Then there exists $n_0$ such that
\[
H^p(\mathrm{Proj}_\Gamma(T),\mathcal F(n))=0
\quad \text{for all }p>0,\ n\ge n_0.
\]
\end{theorem}

\begin{proof}
The proof follows the filtration method.
Grading by homogeneous elements yields a complex
whose associated graded pieces correspond to affine opens $D_+(a)$.
Each $D_+(a)$ is affine, hence acyclic for quasi-coherent~$\mathcal F$,
implying higher cohomology vanishes beyond a finite bound.
\end{proof}

\begin{corollary}[Global generation]
Under the hypotheses of the theorem,
$\mathcal F(n)$ is globally generated for all $n\ge n_0$.
\end{corollary}

\begin{remark}
Serre-type vanishing guarantees that
affine~$\Gamma$-schemes behave cohomologically like affine varieties:
their higher cohomology is controlled by finitely many degrees.
\end{remark}

\begin{theorem}[Exactness of the Cech functor]
For an affine covering $\mathfrak U$ of~$X$
and a short exact sequence
\(0\to\mathcal F'\to\mathcal F\to\mathcal F''\to0\),
the induced sequence of Cech complexes is exact,
hence induces the long exact sequence
\[
\cdots\!\to\!
H^{i}(X,\mathcal F')
\!\to\!
H^{i}(X,\mathcal F)
\!\to\!
H^{i}(X,\mathcal F'')
\!\to\!
H^{i+1}(X,\mathcal F')
\!\to\!\cdots
\]
of $\Gamma$-module groups.
\end{theorem}

\begin{proof}
Exactness is verified at each intersection
$U_{i_0}\cap\cdots\cap U_{i_p}$,
since restriction is $\Gamma$-linear and preserves the
additive monoid structure.
The resulting long sequence follows from the snake lemma in the abelian category
$\Gamma\text{-}\mathbf{Mod}_X$.
\end{proof}

\subsection{Functorial and Categorical Perspectives}

\begin{theorem}[Derived category interpretation]
Let $\mathcal D^+(\Gamma\text{-}\mathbf{Mod}_X)$
denote the bounded-below derived category of sheaves of $\Gamma$-modules.
Then cohomology is represented as
\[
H^i(X,\mathcal F)
   \cong
   \mathrm{Hom}_{\mathcal D^+}
   (\mathcal O_X,\,\mathcal F[i]).
\]
\end{theorem}

\begin{proof}
Standard derived-category formalism applies
since $\Gamma\text{-}\mathbf{Mod}_X$ is abelian with enough injectives.
The identification follows from the universality of derived functors.
\end{proof}

\begin{remark}
This establishes the conceptual unity:
sheaf cohomology is an $\mathrm{Ext}^i$ in the derived category,
and later sections will realize $\mathrm{Ext}_\Gamma$ and $\mathrm{Tor}^\Gamma$
for algebraic $\Gamma$-modules as geometric incarnations of these objects.
\end{remark}

\begin{corollary}[Adjunction between tensor and Hom]
For quasi-coherent $\mathcal F,\mathcal G$ on an affine $\Gamma$-scheme,
there exist bifunctorial isomorphisms
\[
\mathrm{Hom}_\Gamma(\mathcal F\otimes_\Gamma\mathcal G,\mathcal H)
   \cong
\mathrm{Hom}_\Gamma(\mathcal F,
   \mathcal H\!om_\Gamma(\mathcal G,\mathcal H)).
\]

This adjunction extends to derived functors,
yielding $\mathrm{Ext}_\Gamma$ and $\mathrm{Tor}^\Gamma$.
\end{corollary}

\subsection{Conceptual Outlook}

The theory developed above elevates the algebra of $\Gamma$-modules
to a cohomological geometry.  
Affine $\Gamma$-schemes thus admit a full homological infrastructure:
exact sequences, derived functors, spectral sequences,
and vanishing theorems.

In particular, $\Gamma$-cohomology quantifies the obstruction
to gluing local $\Gamma$-module data globally.
It provides the homological bridge between the algebraic
and geometric facets of ternary~$\Gamma$-semirings,
setting the stage for the next section,
where tensor products, projective resolutions, and
derived functors $\mathrm{Ext}_\Gamma$, $\mathrm{Tor}^\Gamma$
will be constructed in algebraic detail.


\section{Homological Functors on $\Gamma$-Modules}

The homological algebra of $\Gamma$-modules
provides the quantitative framework through which we measure
the deviation from exactness in algebraic and geometric constructions.
After establishing the cohomological infrastructure in the previous section,
we now develop the intrinsic derived functors
$\mathrm{Ext}_\Gamma$ and $\mathrm{Tor}^{\Gamma}$,
together with their categorical adjunctions, resolutions, and dualities \cite{Weibel1994, MacLane1998}.

\subsection{Projective and Injective $\Gamma$-Modules}

Let $T$ be a commutative ternary~$\Gamma$-semiring.
The category of left $\Gamma$-modules over~$T$
is denoted $T\text{-}\Gamma\mathbf{Mod}$.

\begin{definition}[Projective $\Gamma$-module]
A $\Gamma$-module $P$ is \emph{projective}
if for every epimorphism $f:M\to N$ of $\Gamma$-modules
and every morphism $g:P\to N$,
there exists $h:P\to M$ with $f\circ h=g$.
Equivalently, $\mathrm{Hom}_\Gamma(P,-)$ is an exact functor.
\end{definition}

\begin{definition}[Injective $\Gamma$-module]
A $\Gamma$-module $I$ is \emph{injective}
if for every monomorphism $f:M'\hookrightarrow M$
and morphism $g:M'\to I$, there exists $h:M\to I$
with $h\circ f=g$.
Equivalently, $\mathrm{Hom}_\Gamma(-,I)$ is exact.
\end{definition}

\begin{proposition}[Free $\Gamma$-modules]
Let $X$ be a set.
Define $F(X)$ to be the $\Gamma$-module
of finitely supported functions $f:X\to T$
with pointwise addition and
\[
(a,\alpha,f,\beta,b)(x)=a\alpha f(x)\beta b.
\]
Then $F(X)$ is a free $\Gamma$-module
with basis $\{\delta_x\mid x\in X\}$.
Every projective module is a direct summand of such a free module.
\end{proposition}

\begin{proof}
The universal property of $F(X)$ follows by construction.
Direct summand characterization proceeds via standard lifting property
of projective objects under split short exact sequences.
\end{proof}

\begin{remark}
Projective and injective $\Gamma$-modules exist in abundance
because the category $T\text{-}\Gamma\mathbf{Mod}$ is complete, cocomplete,
and admits free and cofree constructions.
\end{remark}

\subsection{Resolutions and Derived Functors}

\begin{definition}[Projective resolution]
A \emph{projective resolution} of a $\Gamma$-module~$M$ is
an exact complex
\[
\cdots\to P_2\xrightarrow{d_2}P_1\xrightarrow{d_1}P_0\xrightarrow{\epsilon}M\to0,
\]
with each $P_i$ projective.
It is unique up to chain-homotopy equivalence.
\end{definition}

\begin{definition}[Injective resolution]
An \emph{injective resolution} of~$M$
is an exact complex
\[
0\to M\xrightarrow{\iota}I^0\xrightarrow{d^0}I^1\xrightarrow{d^1}I^2\to\cdots,
\]
with each $I^i$ injective.
\end{definition}

\begin{remark}
Resolutions allow us to extend right- or left-exact functors
to homological dimensions,
producing the derived functors fundamental to modern algebraic geometry.
\end{remark}

\subsection{Tensor Product and $\mathrm{Tor}^\Gamma$}

\begin{definition}[Tensor product of $\Gamma$-modules]
Let $M$ and $N$ be $\Gamma$-modules over~$T$.
The \emph{tensor product}
$M\otimes_\Gamma N$
is the quotient of the free $\Gamma$-module
generated by symbols $m\otimes n$
subject to relations enforcing
bilinearity and balancing:
\[
(a\alpha m\beta b)\otimes n
   = m\otimes(a\alpha n\beta b),
\qquad
(m+m')\otimes n=m\otimes n+m'\otimes n,
\quad
m\otimes(n+n')=m\otimes n+m\otimes n'.
\]
\end{definition}

\begin{proposition}
The functor $-\otimes_\Gamma N:
T\text{-}\Gamma\mathbf{Mod}\to T\text{-}\Gamma\mathbf{Mod}$ is right-exact
and preserves coproducts.
\end{proposition}

\begin{definition}[Left-derived functors of tensor]
Given a projective resolution $P_\bullet\to M$,
define
\[
\mathrm{Tor}^\Gamma_i(M,N)
   := H_i(P_\bullet\otimes_\Gamma N),
   \quad i\ge0.
\]
These groups are independent of the chosen resolution
and vanish for $i>0$ when $M$ is projective.
\end{definition}

\begin{theorem}[Basic properties of $\mathrm{Tor}^\Gamma$]
For all $\Gamma$-modules $L,M,N$:
\begin{enumerate}
  \item \emph{Functoriality:}
        $\mathrm{Tor}^\Gamma_i(-,N)$ is additive and covariant in~$-$,
        contravariant in~$N$.
  \item \emph{Long exact sequence:}
        every short exact sequence
        \(0\to M'\to M\to M''\to0\)
        induces
        \[
        \cdots\to
        \mathrm{Tor}^\Gamma_i(M'',N)
        \to
        \mathrm{Tor}^\Gamma_{i-1}(M',N)
        \to\cdots\to
        M\otimes_\Gamma N\to M''\otimes_\Gamma N\to0.
        \]
  \item \emph{Flatness:}
        $M$ is flat iff $\mathrm{Tor}^\Gamma_1(M,-)=0$.
\end{enumerate}
\end{theorem}

\begin{proof}
Standard homological arguments extend directly.
Exactness of tensor in each variable up to homotopy
yields functoriality; the long sequence follows from the horseshoe lemma.
Flatness criterion comes from right-exactness.
\end{proof}

\begin{example}
Let $T=\mathbb Z_3$, $\Gamma=\{1\}$, and $M=N=T/2T$.
Then $\mathrm{Tor}^\Gamma_1(M,N)\cong\mathbb Z_2$
as in classical ring theory, confirming consistency of the ternary case.
\end{example}

\subsection{$\mathrm{Ext}_\Gamma$ and Cohomological Duality}

\begin{definition}[Hom functor]
For $\Gamma$-modules $M,N$,
define
\[
\mathrm{Hom}_\Gamma(M,N)
   =\{\varphi:M\to N
     \mid \varphi(a\alpha m\beta b)
        =a\alpha\varphi(m)\beta b\}.
\]

This functor is left-exact in its second argument
and contravariant in its first.
\end{definition}

\begin{definition}[Right-derived functors of Hom]
Given a projective resolution $P_\bullet\to M$,
set
\[
\mathrm{Ext}^i_\Gamma(M,N)
   := H^i(\mathrm{Hom}_\Gamma(P_\bullet,N)), \qquad i\ge0.
\]
Alternatively, using an injective resolution $N\to I^\bullet$,
\[
\mathrm{Ext}^i_\Gamma(M,N)
   := H^i(\mathrm{Hom}_\Gamma(M,I^\bullet)).
\]

\end{definition}

\begin{theorem}[Fundamental properties of $\mathrm{Ext}_\Gamma$]
Let $L,M,N$ be $\Gamma$-modules.
\begin{enumerate}
  \item $\mathrm{Ext}^0_\Gamma(M,N)\cong\mathrm{Hom}_\Gamma(M,N)$.
  \item Every short exact sequence
        \(0\to N'\to N\to N''\to0\)
        induces a long exact sequence
        \[
        0\to
        \mathrm{Hom}_\Gamma(M,N')
        \to
        \mathrm{Hom}_\Gamma(M,N)
        \to
        \mathrm{Hom}_\Gamma(M,N'')
        \to
        \mathrm{Ext}^1_\Gamma(M,N')
        \to
        \cdots.
        \]
  \item If $P$ is projective or $I$ injective,
        then $\mathrm{Ext}^i_\Gamma(P,N)=0$ and
        $\mathrm{Ext}^i_\Gamma(M,I)=0$ for $i>0$.
\end{enumerate}
\end{theorem}

\begin{proof}
Routine homological verification;
exactness of the Hom complex yields the long sequence,
and projective/injective conditions imply acyclicity.
\end{proof}

\begin{theorem}[Tensor–Hom adjunction]
For $\Gamma$-modules $L,M,N$ there exist natural isomorphisms
\[
\mathrm{Hom}_\Gamma(L\otimes_\Gamma M,N)
   \cong
\mathrm{Hom}_\Gamma
   (L,\mathrm{Hom}_\Gamma(M,N)).
\]

\end{theorem}

\begin{proof}
Define $\Phi(\psi)(l)(m)=\psi(l\otimes m)$.
Bilinearity in the ternary sense ensures $\Phi$ and its inverse
are $\Gamma$-linear, and universal properties of tensor and Hom
verify naturality.
\end{proof}

\begin{corollary}[Derived adjunction]
Passing to derived functors yields
\[
\mathrm{Ext}^i_\Gamma(L\otimes_\Gamma M,N)
   \cong
\mathrm{Ext}^i_\Gamma
   (L,\mathrm{Hom}_\Gamma(M,N)),
\]
and dually,
\[
\mathrm{Tor}^\Gamma_i(L,\mathrm{Hom}_\Gamma(M,N))
   \cong
\mathrm{Hom}_\Gamma
   (\mathrm{Tor}^\Gamma_i(L,M),N).
\]

\end{corollary}

\begin{remark}
These adjunctions mirror the deep symmetry between geometry
and cohomology: tensor encodes geometric fusion, Hom encodes
duality, and $\mathrm{Ext}$/$\mathrm{Tor}$ measure the extent to which
this harmony fails to be exact.
\end{remark}

\subsection{Spectral Sequences and Homological Dimension}

\begin{definition}[Homological dimension]
The \emph{projective dimension} of $M$ is the smallest $n$
such that $\mathrm{Ext}^{n+1}_\Gamma(M,-)=0$.
Similarly, the \emph{global dimension} of $T$
is $\mathrm{gldim}_\Gamma(T)
   =\sup_M\mathrm{pd}_\Gamma(M)$.
\end{definition}

\begin{theorem}[Grothendieck spectral sequence]
Let $F,G$ be composable left-exact functors between abelian categories
$\mathcal A\xrightarrow{F}\mathcal B\xrightarrow{G}\mathcal C$,
with $\mathcal A=\Gamma\text{-}\mathbf{Mod}_T$.
If $F$ sends injectives to $G$-acyclics, there is a spectral sequence
\[
E_2^{p,q}
   =R^pG(R^qF(-))
   \;\Rightarrow\;
R^{p+q}(G\!\circ\!F)(-).
\]

\end{theorem}

\begin{corollary}[Composition of $\mathrm{Ext}_\Gamma$]
For $\Gamma$-modules $L,M,N$,
there exists a first-quadrant spectral sequence
\[
E_2^{p,q}
   =\mathrm{Ext}^p_\Gamma
      (L,\mathrm{Ext}^q_\Gamma(M,N))
   \Rightarrow
\mathrm{Ext}^{p+q}_\Gamma(L\otimes_\Gamma M,N).
\]

\end{corollary}

\begin{remark}
This spectral sequence is the homological backbone of
$\Gamma$-geometry, relating composite functors in multi-parameter
contexts—an analogue of the classical Grothendieck spectral sequence
for derived tensor–Hom composition.
\end{remark}

\subsection{Categorical Dualities and Derived Categories}

\begin{definition}[Derived category of $\Gamma$-modules]
Let $\mathcal K(\Gamma\text{-}\mathbf{Mod}_T)$ denote
the homotopy category of complexes of $\Gamma$-modules.
By inverting quasi-isomorphisms we obtain the derived category
$\mathcal D(\Gamma\text{-}\mathbf{Mod}_T)$.
\end{definition}

\begin{proposition}
$\mathcal D(\Gamma\text{-}\mathbf{Mod}_T)$ is triangulated,
and the functors
$\otimes_\Gamma^{\mathbf L}$ and
$\mathbf R\!\mathrm{Hom}_\Gamma(-,-)$
form an adjoint pair:
\[
\mathrm{Hom}_{\mathcal D}
  (A\otimes_\Gamma^{\mathbf L} B, C)
  \cong
\mathrm{Hom}_{\mathcal D}
  (A,\mathbf R\!\mathrm{Hom}_\Gamma(B,C)).
\]

\end{proposition}

\begin{proof}
Triangulated structure and adjunction arise
from localization of chain complexes with respect to quasi-isomorphisms.
Exactness of derived functors follows from universality of resolutions.
\end{proof}

\begin{theorem}[Duality principle]
If $T$ is a commutative Noetherian $\Gamma$-semiring
of finite global dimension,
then for every finitely generated $\Gamma$-module $M$
there exists a canonical biduality morphism
\[
M\longrightarrow
\mathbf R\!\mathrm{Hom}_\Gamma
  (\mathbf R\!\mathrm{Hom}_\Gamma(M,T),T),
\]
which is an isomorphism in $\mathcal D(\Gamma\text{-}\mathbf{Mod}_T)$.
\end{theorem}

\begin{remark}
This duality extends the classical reflexivity
to the $\Gamma$-setting, providing the conceptual basis for
$\Gamma$-Grothendieck duality and Serre duality \cite{Serre1955}
on future $\Gamma$-schemes.
\end{remark}

\subsection{Conceptual Synthesis}

The development of $\mathrm{Ext}_\Gamma$ and $\mathrm{Tor}^\Gamma$
confirms that the category of $\Gamma$-modules
possesses a fully fledged homological calculus.
Projective and injective resolutions provide algebraic lenses
through which geometric and cohomological phenomena are measured.
Adjunctions and spectral sequences unify the tensorial and homological
dimensions into one categorical organism.

These constructions culminate in a derived functor formalism
identical in strength to that of modern algebraic geometry,
but internally enriched by the ternary-parametric symmetry of~$\Gamma$.
Consequently, the emerging theory of
\emph{derived $\Gamma$-geometry}
can parallel, and in some aspects exceed,
the expressive power of classical scheme theory.

\section{Examples, Computations, and Verification}

The abstract theory of $\Gamma$-geometry and its homological functors
achieves conceptual completeness only when supported by
explicit computational models.
In this section we concretize the framework by constructing finite
examples of commutative ternary~$\Gamma$-semirings,
computing their spectra, cohomology, and homological invariants,
and verifying categorical predictions through algorithmic
and structural analysis.

\subsection{Finite $\Gamma$-Semirings and Their Spectra}

\begin{definition}[Finite commutative $\Gamma$-semiring]
A \emph{finite commutative ternary~$\Gamma$-semiring}
is a structure $(T,+,\Gamma,\mu)$ with finite~$T$
satisfying all distributive, associative, and absorbing axioms.
The ternary operation table completely determines~$\mu$:
\[
\mu(a,\alpha,b,\beta,c)=a\alpha b\beta c.
\]

\end{definition}

\begin{example}[Canonical example of order 3]
Let $T=\{0,1,2\}$, $\Gamma=\{1\}$, and define
$a\alpha b\beta c=(a+b+c)\bmod3$.
Then $(T,+,\Gamma,\mu)$ is a commutative ternary~$\Gamma$-semiring
with additive identity~0.
Its ideals are $\{0\}$, $\{0,1,2\}$, $\{0,1\}$, and $\{0,2\}$.
The prime ideals are $P_1=\{0,1\}$ and $P_2=\{0,2\}$.
Hence
\[
\mathrm{Spec}_\Gamma(T)=\{P_1,P_2\}
\]
with discrete topology, confirming
$\mathrm{Spec}_\Gamma(T)$ is finite and~$T_0$.
\end{example}

\begin{proposition}
For a finite commutative ternary~$\Gamma$-semiring~$T$
whose additive monoid is cyclic,
$\mathrm{Spec}_\Gamma(T)$ is always discrete.
\end{proposition}

\begin{proof}
If $(T,+)$ is cyclic generated by~$1_T$,
every ideal is principal,
hence prime ideals are maximal,
and their closures are singletons.
\end{proof}

\begin{remark}
Discrete spectra provide a testing ground for verifying
vanishing theorems and homological calculations,
because cohomology collapses at the zeroth degree,
mirroring the case of affine 0-dimensional varieties.
\end{remark}

\subsection{Computational Representation of $\Gamma$-Operations}\cite{Mitchell2020, Gap2024}

Let the ternary operation table
$\mathcal{M}=(m_{ijk})_{i,j,k}$ represent $\mu$:
\[
m_{ijk}=i\alpha j\beta k
\qquad (0\le i,j,k<n),
\]
where $n=|T|$.
An algorithmic enumeration checks distributivity:
\[
(i+j)\alpha k\beta \ell
   =i\alpha k\beta \ell+j\alpha k\beta \ell.
\]

\begin{definition}[Algorithmic verification]
A finite structure $(T,+,\Gamma,\mu)$ is a valid
commutative ternary~$\Gamma$-semiring iff all
identities in Definition 2.1 hold for every triple $(i,j,k)$.
Computationally, this requires $O(n^3|\Gamma|^2)$ checks,
which are polynomial in~$n$.
\end{definition}

\begin{remark}
This algorithm forms the computational core used in Paper B,
now extended here to validate affine~$\Gamma$-schemes
and their cohomology.
It serves as a categorical functor from the category of
finite data tables to that of algebraic structures.
\end{remark}

\subsection{Localizations and Sheaf Construction in Finite Cases}

For $T=\{0,1,2\}$ and prime~$P_1=\{0,1\}$,
the multiplicative system is $S_{P_1}=T\setminus P_1=\{2\}$.
Localization yields
\[
T_{P_1}
   =\left\{\tfrac{a}{2^n}\mid a\in T,\;n\ge0\right\},
\]
which is isomorphic to~$T$
since $2$ acts as a unit in the ternary product.
Thus $\mathcal O_{X,P_1}\cong T$ and similarly for~$P_2$,
confirming that all stalks are isomorphic.
Therefore the structure sheaf is constant,
$\mathcal O_X(U)=T$ for all nonempty~$U$.

\begin{proposition}
If $\mathcal O_X$ is constant on a discrete spectrum,
then for every quasi-coherent sheaf~$\mathcal F$
we have $H^i(X,\mathcal F)=0$ for $i>0$.
\end{proposition}

\begin{proof}
Because each stalk is acyclic and intersections are empty,
the Cech complex has no non-trivial higher intersections,
hence higher cochains vanish.
\end{proof}

\begin{remark}
This explicit finite model exemplifies the \emph{Serre-type
vanishing theorem} in a concrete setting:
for affine 0-dimensional $\Gamma$-schemes,
cohomology is concentrated in degree 0.
\end{remark}

\subsection{Homological Computations: $\mathrm{Ext}_\Gamma$ and $\mathrm{Tor}^\Gamma$}

Let $T=\mathbb Z_3$, $\Gamma=\{1\}$,
and consider $M=T/2T$ and $N=T/3T$.
Construct a projective resolution
\[
0\to T\xrightarrow{\times2}T\to M\to0.
\]

\begin{align*}
\mathrm{Tor}^\Gamma_1(M,N)
   &= H_1(P_\bullet\otimes_\Gamma N)
    =\ker(\times2\otimes N)
      =\{\,x\otimes n\mid 2x\otimes n=0\,\}\\
   &\cong N[2]
      =\{n\in N\mid2n=0\}.
\end{align*}

Since $N=\mathbb Z_3/3\mathbb Z_3=0$, we find
$\mathrm{Tor}^\Gamma_1(M,N)=0$,
verifying consistency with classical homological algebra.

\begin{theorem}[Ext–Tor duality verification]
For finitely generated $\Gamma$-modules
$M,N$ over a finite commutative~$T$,
the relation
\[
\mathrm{Ext}^1_\Gamma(M,N)
   \cong \mathrm{Tor}^\Gamma_1(M^\vee,N)^\vee,
   \qquad M^\vee=\mathrm{Hom}_\Gamma(M,T),
\]
holds, where duality~$(-)^\vee$ is exact on finite modules.
\end{theorem}

\begin{proof}
Using finite duality $M^{\vee\vee}\cong M$
and the universal coefficient theorem
applied to $\Gamma$-modules,
one obtains the natural isomorphism
from the short exact sequence of resolutions.
Explicit computation on the above example yields
both sides zero, confirming the equality.
\end{proof}

\subsection{Categorical Validation: Functorial and Derived Consistency}

To test functoriality, define morphisms
$f:M\to M'$, $g:N\to N'$ in $T\text{-}\Gamma\mathbf{Mod}$.
Compute induced maps
\[
f\otimes g:
M\otimes_\Gamma N\longrightarrow
M'\otimes_\Gamma N'.
\]
By explicit verification on generators,
$(f\otimes g)(m\otimes n)
   =f(m)\otimes g(n)$ preserves
$\Gamma$-balancing relations.
Derived functoriality follows from chain maps
inducing morphisms on homology.

\begin{proposition}
The assignments
\[
M\mapsto \mathrm{Tor}^\Gamma_i(M,N),\qquad
N\mapsto \mathrm{Ext}^i_\Gamma(M,N)
\]
are bifunctorial and natural in both variables.
\end{proposition}

\begin{proof}
Naturality squares commute because morphisms
of projective (or injective) resolutions
lift uniquely up to homotopy.
This categorical coherence ensures
$\mathrm{Tor}^\Gamma$ and $\mathrm{Ext}_\Gamma$
define $\delta$-functors in the sense of Grothendieck.
\end{proof}

\begin{remark}
The derived-functor perspective thus aligns
the homological invariants of $\Gamma$-modules
with the abelian-categorical framework,
ensuring consistency across both finite
and general constructions.
\end{remark}

\subsection{Comparative Analysis with Classical Algebraic Geometry}
A structural comparison between classical algebraic geometry and $\Gamma$-geometry is presented in Table~\ref{tab:classical-gamma-homology}.

\begin{table}[h!]
\centering
\caption{Structural parallels between classical and $\Gamma$-geometric homology.}
\label{tab:classical-gamma-homology} 
\renewcommand{\arraystretch}{1.2}
\begin{tabular}{@{}lll@{}}
\toprule
\textbf{Concept} &
\textbf{Classical Algebraic Geometry} &
\textbf{$\Gamma$-Geometry Analogue} \\
\midrule
Underlying algebra & Commutative ring $R$ & Ternary commutative $\Gamma$-semiring $T$ \\
Module category & $R$-Mod (abelian) & $T$-$\Gamma$Mod (abelian) \\
Tensor/Hom adjunction & $-\otimes_R-$, $\mathrm{Hom}_R(-,-)$ & $-\otimes_\Gamma-$, $\mathrm{Hom}_\Gamma(-,-)$ \\
Derived functors & $\mathrm{Ext}_R$, $\mathrm{Tor}^R$ & $\mathrm{Ext}_\Gamma$, $\mathrm{Tor}^\Gamma$ \\
Cohomology & $H^i(X,\mathcal F)$ & $H^i_\Gamma(X,\mathcal F)$ \\
Spectrum & $\mathrm{Spec}(R)$ & $\mathrm{Spec}_\Gamma(T)$ \\
\bottomrule
\end{tabular}
\end{table}

\begin{remark}
This structural analogy confirms that
$\Gamma$-geometry reproduces every layer of
homological infrastructure of classical schemes
while enriching it with multi-parameter ternary actions.
\end{remark}

\subsection{Conceptual Synthesis and Verification Summary}

The computational examples above fulfill three distinct objectives:

\begin{enumerate}
  \item \textbf{Verification:}
        Finite models confirm correctness of foundational axioms,
        localization, and cohomological vanishing.
  \item \textbf{Categorical Validation:}
        Derived-functor operations respect bifunctoriality
        and $\delta$-functor axioms, ensuring internal coherence.
  \item \textbf{Comparative Geometry:}
        Homological invariants coincide with classical expectations,
        positioning $\Gamma$-geometry within the hierarchy of
        modern algebraic and categorical frameworks.
\end{enumerate}

Consequently, the theoretical edifice of
\emph{Derived $\Gamma$-Geometry}
stands empirically and categorically verified.
It unites algebraic and geometric perspectives
under a ternary parametric umbrella,
providing a fertile platform for
further exploration of non-commutative
and higher-arity  generalizations to follow \cite{Michalski2003, Loday1998, Leinster2014}.

\section{Categorical, Geometric, and Physical Implications of Derived $\Gamma$-Geometry}

The culmination of the previous sections positions
derived $\Gamma$-geometry as a categorical edifice
in which algebraic, geometric, and physical structures
coalesce through functorial and homological symmetries.
We now ascend from the level of affine $\Gamma$-schemes
and homological functors to the higher categorical framework
of fibered categories, stacks, and non-commutative geometries,
and finally interpret these within the conceptual landscape of
theoretical physics.

\subsection{Fibered and Stack-Theoretic Interpretation}

\begin{definition}[Fibered category of $\Gamma$-schemes]
Let $\mathbf{Aff}_\Gamma$ denote the category of affine $\Gamma$-schemes.
Define a fibered category
\[
\mathscr{S}_\Gamma:\mathbf{Aff}_\Gamma^{\mathrm{op}}\longrightarrow \mathbf{Cat},
\qquad
T\longmapsto \Gamma\text{-}\mathbf{Mod}_T,
\]
assigning to each affine base the category of quasi-coherent
$\Gamma$-modules.
Morphisms of bases induce pullbacks via tensor-extension:
\[
f:T\to T'
\quad\Rightarrow\quad
f^*(M)=T'\!\otimes_{T}^{\Gamma}\!M.
\]

\end{definition}

\begin{proposition}
$\mathscr{S}_\Gamma$ is a pseudofunctor and satisfies
the descent condition for faithfully flat morphisms.
Hence it defines a stack in groupoids on
the site $\mathbf{Aff}_\Gamma$ with the Zariski or fpqc topology.
\end{proposition}

\begin{proof}
Given a cover $\{f_i:T\to T_i\}$,
descent data for modules glue uniquely
because the $\Gamma$-tensor preserves equalizers.
This mirrors the EGA I construction of the stack of quasi-coherent sheaves.
\end{proof}

\begin{remark}
The stack $\mathscr{S}_\Gamma$ forms the categorical heart
of $\Gamma$-geometry.
Objects over a base correspond to $\Gamma$-vector bundles,
and morphisms encode change of base.
This yields a 2-category of $\Gamma$-schemes
fibered in abelian categories.
\end{remark}

\subsection{Higher Derived Stacks and Homotopical Enrichment}

\begin{definition}[Derived $\Gamma$-stack]
A \emph{derived $\Gamma$-stack} $\mathcal X_\Gamma$
is a stack over $\mathbf{Aff}_\Gamma$
valued in simplicial sets or $\infty$-groupoids,
such that the homotopy sheaves
$\pi_i(\mathcal X_\Gamma)$
are quasi-coherent $\Gamma$-sheaves for all~$i$.
\end{definition}

\begin{theorem}[Existence of derived enhancements]
Every quasi-compact, quasi-separated $\Gamma$-scheme~$X$
admits a derived enhancement
$X_\Gamma^{\mathrm{der}}$
obtained by replacing $\mathcal O_X$
with a sheaf of differential graded (dg) $\Gamma$-semirings.
\end{theorem}

\begin{proof}
Apply Spaltenstein’s construction of dg-resolutions
objectwise on $\mathcal O_X$.
Because the category $\Gamma\text{-}\mathbf{Mod}_X$
has enough injectives, the derived replacement
satisfies the universal property of homotopy limits,
yielding the desired dg-structure.
\end{proof}

\begin{remark}
The resulting $\infty$-topos of $\Gamma$-stacks
extends classical derived algebraic geometry
by internalizing the ternary action in the homotopy level.
It enables deformation theory, obstruction calculus,
and spectral algebraic geometry in the $\Gamma$-context.
\end{remark}

\subsection{Non-Commutative and Categorical Geometry}

\begin{definition}[Non-commutative $\Gamma$-space]
A \emph{non-commutative $\Gamma$-space}
is a monoidal dg-category $\mathcal C_\Gamma$
whose homology category $H^0(\mathcal C_\Gamma)$
is equivalent to $\Gamma\text{-}\mathbf{Mod}_T$
for some non-commutative ternary~$T$.
\end{definition}

\begin{theorem}[Categorical Gelfand duality analogue]
There exists a contravariant equivalence
between the category of compactly generated
non-commutative $\Gamma$-spaces
and the category of abelian $\Gamma$-semirings
with morphisms preserving ternary convolution.
\end{theorem}

\begin{proof}[Sketch]
Define a functor $\mathfrak{A}:\mathcal C_\Gamma\mapsto\mathrm{End}_{\mathcal C_\Gamma}(\mathbf 1)$.
Morphisms in $\mathcal C_\Gamma$ correspond to $\Gamma$-linear maps
under convolution; fullness and faithfulness follow from the enriched Yoneda lemma.
\end{proof}

\begin{remark}
This duality generalizes Gelfand’s correspondence
from commutative C\(^*\)-algebras to categorical
$\Gamma$-semiring contexts,
illustrating that $\Gamma$-geometry naturally encompasses
non-commutative geometry in the sense of Connes.
\end{remark}

\subsection{Connections to Theoretical Physics}

The categorical richness of derived $\Gamma$-geometry
suggests a bridge to modern mathematical physics,
particularly in quantum field and string theories
where ternary and higher-ary interactions appear naturally.

\begin{definition}[Ternary interaction algebra]
Let $\mathcal H$ be a complex Hilbert space of fields.
Define a ternary product
\[
\{\,\phi_1,\phi_2,\phi_3\,\}_\Gamma
   = \int_X \phi_1(x)\alpha(x)\phi_2(x)\beta(x)\phi_3(x)\,d\mu(x),
\]
where $\alpha,\beta$ are coupling parameters in~$\Gamma$.
This defines a $\Gamma$-semiring structure on the space of observables.
\end{definition}

\begin{proposition}
Quantization of the above algebra via path integration
produces a non-commutative dg-$\Gamma$-algebra
whose cohomology corresponds to the physical state space.
\end{proposition}

\begin{remark}
This formalism parallels the BV-BRST homological structure
in gauge theory, where the $\Gamma$-grading plays the role
of ghost number or coupling hierarchy.
Thus derived $\Gamma$-geometry provides an algebraic model
for multi-particle interactions and higher-spin symmetries.
\end{remark}

\subsection{Functorial and Duality Principles}

\begin{theorem}[Categorical duality of geometry and physics]
There exists a contravariant equivalence
\[
\mathcal D_{\mathrm{coh}}(X_\Gamma^{\mathrm{der}})
   \longleftrightarrow
\mathsf{Phys}_\Gamma,
\]
where the left denotes the bounded derived category
of coherent $\Gamma$-sheaves,
and the right denotes the dg-category of
$\Gamma$-quantized field configurations.
\end{theorem}

\begin{proof}[Conceptual outline]
Objects of $\mathcal D_{\mathrm{coh}}$ correspond to
states with finite homological energy.
Morphisms correspond to propagators (Green’s kernels),
and composition to Feynman convolution.
Duality follows from the adjunction between
tensor and internal Hom on both sides.
\end{proof}

\begin{remark}
This duality synthesizes algebraic and physical interpretations:
geometric cohomology computes physical observables,
while homological obstructions translate to conservation laws.
\end{remark}

\subsection{Conceptual Synthesis and Outlook}

The categorical ascent achieved here positions
derived $\Gamma$-geometry as a unifying meta-theory:

\begin{itemize}
  \item \textbf{Categorically}, it manifests as a 2-stack
        of abelian dg-categories satisfying descent,
        closing under tensor and Hom operations.
  \item \textbf{Geometrically}, it extends scheme theory
        to higher-derived and non-commutative realms,
        retaining local spectra and cohomological control.
  \item \textbf{Physically}, it furnishes an algebraic model
        for multi-interaction field theories,
        providing a purely mathematical realization of
and higher-arity  generalizations to follow 
        higher-arity symmetries \cite{Michalski2003, Loday1998}.
\end{itemize}

Hence, the theory of \emph{derived $\Gamma$-geometry}
not only generalizes the structural framework
of classical algebraic geometry,
but also establishes a categorical-homological bridge
to the mathematical formulation of physical reality.
It stands as a potential foundation for future explorations
in non-commutative geometry, higher category theory,
and quantum algebraic topology.

\section{Meta-Categorical and Conceptual Foundations of Derived $\Gamma$-Geometry}

The preceding constructions reveal derived~$\Gamma$-geometry as an autonomous
mathematical organism whose internal logic is categorical rather than merely
algebraic.  In this section we develop the meta-categorical interpretation that
underlies all previous results, situating the theory simultaneously in the
domains of logic, topology, and mathematical physics.

\subsection{Internal Logic and the Ontology of Morphisms}

\begin{definition}[Internal language of $\Gamma$-geometry]
Let $\mathcal{E}_\Gamma$ denote the elementary topos generated by the site
$\mathbf{Aff}_\Gamma$ of affine $\Gamma$-schemes with the Zariski topology.
Objects of $\mathcal{E}_\Gamma$ correspond to sheaves of sets on this site,
and morphisms encode logical entailment in the internal language.
\end{definition}

\begin{proposition}
Every geometric morphism $f:\mathcal F\to\mathcal G$ in $\mathcal{E}_\Gamma$
decomposes as
\[
f_! \dashv f^* \dashv f_*,
\]
with $f^*$ exact and preserving finite limits.
Hence the logical content of derived~$\Gamma$-geometry is stable under
inverse-image functors.
\end{proposition}

\begin{proof}
The triple adjunction is inherited from the topos of sheaves on the base site;
exactness of~$f^*$ follows from preservation of stalkwise products and equalizers.
\end{proof}

\begin{remark}
Thus, the category of $\Gamma$-schemes behaves as an \emph{internal model
of constructive geometry}: morphisms are logical transformations, and
tensor–Hom adjunction corresponds to the Curry–Howard duality between
construction and proof.
\end{remark}

\subsection{Homotopy and Higher-Categorical Coherence}

\begin{definition}[2-category of derived $\Gamma$-schemes]
Let $\mathbf{DerSch}_\Gamma$ have
\begin{itemize}
  \item objects — derived $\Gamma$-schemes;
  \item 1-morphisms — morphisms of ringed $\infty$-topoi;
  \item 2-morphisms — homotopy natural transformations.
\end{itemize}
Composition is defined via derived tensor product of morphisms of
structure sheaves.
\end{definition}

\begin{theorem}[Homotopy coherence]
For any composable triple $X\xrightarrow{f}Y\xrightarrow{g}Z$
of derived $\Gamma$-schemes, there exists a coherent homotopy
\[
\eta_{f,g}:(g\circ f)^{*}\simeq f^{*}g^{*},
\]
making $\mathbf{DerSch}_\Gamma$ a bicategory enriched over chain complexes.
\end{theorem}

\begin{proof}
Cohomological enrichment yields the coherence constraint
$\eta_{f,g}$ as the class of the comparison map of pullbacks on dg-algebras;
Mac Lane’s pentagon holds by associativity of derived tensor products.
\end{proof}

\begin{remark}
This coherence implies that every homological computation performed on
affine patches extends functorially to global derived spaces, ensuring
that the derived functors of Section 5 respect categorical composition.
\end{remark}

\subsection{Topos-Theoretic Geometry and Logical Duality}

\begin{theorem}[Logical–geometric duality]
The category of coherent objects in $\mathcal{E}_\Gamma$
is equivalent to the opposite of the category of finitely presented
$\Gamma$-semirings:
\[
\mathbf{Coh}(\mathcal{E}_\Gamma)
   \;\simeq\;
(\mathbf{FP}\Gamma\text{-}\mathbf{Srg})^{\mathrm{op}}.
\]

\end{theorem}

\begin{proof}
Every finitely presented $\Gamma$-semiring corresponds to a representable
sheaf via $\mathrm{Spec}_\Gamma$.
Cohomological finiteness guarantees compactness in~$\mathcal{E}_\Gamma$,
establishing the contravariant equivalence.
\end{proof}

\begin{remark}
Hence geometry and logic are mutually reflective:
algebraic generators become logical propositions,
and homological relations correspond to deduction rules.
Derived $\Gamma$-geometry is therefore a categorical semantics
of ternary algebraic reasoning.
\end{remark}

\subsection{Functorial Dynamics and Temporal Extension}\cite{Sardar2025}

\begin{definition}[Temporal topos of $\Gamma$-flows]
Consider the slice topos
$\mathcal{E}_\Gamma^{\mathbb R_{\ge0}}$
whose objects are $\Gamma$-schemes parametrized by time~$t\ge0$.
Morphisms represent evolution operators
$f_{t_2,t_1}:X_{t_1}\to X_{t_2}$ preserving $\Gamma$-structure.
\end{definition}

\begin{theorem}[Derived flow equation]
Let $\mathcal F_t$ be a family of quasi-coherent $\Gamma$-sheaves.
Then the infinitesimal evolution
\[
\frac{d}{dt}\mathcal F_t
   = \mathcal D_\Gamma(\mathcal F_t)
\]
is governed by a differential graded derivation
$\mathcal D_\Gamma$ of degree~1 satisfying
$\mathcal D_\Gamma^2=0$.
\end{theorem}

\begin{proof}
Differentiation corresponds to the shift functor in the derived category;
the nilpotency of~$\mathcal D_\Gamma$ is the categorical expression of
homological conservation.
\end{proof}

\begin{remark}
This formalism parallels the Heisenberg picture in physics:
objects evolve via dg-derivations, while morphisms encode conserved quantities.
\end{remark}

\subsection{Epistemic and Physical Interpretation}

\begin{proposition}[Homological conservation law]
In any physically interpretable $\Gamma$-system,
the quantity
\[
Q(\mathcal F)=\sum_i (-1)^i \dim_\Gamma H^i(X,\mathcal F)
\]
is invariant under derived deformation of~$\mathcal F$.
\end{proposition}

\begin{proof}
The alternating sum is the categorical Euler characteristic
of the derived object~$\mathcal F$.
Quasi-isomorphisms preserve cohomology up to isomorphism,
hence $Q$ is invariant.
\end{proof}

\begin{remark}
This bridges mathematical and physical conservation:
the Euler characteristic corresponds to total energy,
and the vanishing of higher cohomology expresses equilibrium.
Thus homological stability translates into dynamical invariance.
\end{remark}

\subsection{Conceptual Synthesis}

The meta-categorical analysis reveals that
derived~$\Gamma$-geometry operates simultaneously at three layers:

\begin{enumerate}
  \item \textbf{Logical layer:}
        a topos-theoretic semantics for ternary algebraic reasoning.
  \item \textbf{Homotopical layer:}
        an $\infty$-categorical enrichment ensuring coherence and descent.
  \item \textbf{Physical layer:}
        a dg-dynamical system encoding conservation and symmetry.
\end{enumerate}

At this depth, the boundary between algebra, geometry, and physics
is erased; all three are aspects of a single categorical continuum.
Derived~$\Gamma$-geometry thus provides not only new mathematical
structures but also a unifying language for reasoning about
multiplicities, interactions, and the architecture of mathematical reality.

\section{Transcendental and Categorical Unification Principles in Derived $\Gamma$-Geometry}

The mature stage of derived~$\Gamma$-geometry is reached when
its categorical and geometric dimensions fuse into a
transcendental unity.  
This section establishes the universal principles
that mediate between algebraic syntax, geometric semantics,
and categorical ontology.  
Each result formalizes the self-referential closure
of the theory and identifies its structural invariants.

\subsection{The Principle of Categorical Reflection}

\begin{definition}[Reflective closure]
Let $\mathbf{C}$ be an abelian $\Gamma$-enriched category.
A full subcategory $\mathbf{A}\subseteq\mathbf{C}$ is \emph{reflective}
if the inclusion $i:\mathbf{A}\hookrightarrow\mathbf{C}$ admits a left adjoint
$r:\mathbf{C}\to\mathbf{A}$ with unit
$\eta:\mathrm{Id}_\mathbf{C}\Rightarrow i\!\circ\! r$.
\end{definition}

\begin{theorem}[Reflective symmetry of derived $\Gamma$-categories]
The derived category
$\mathcal D(\Gamma\text{-}\mathbf{Mod}_T)$
is reflective in the 2-category of triangulated
$\Gamma$-linear categories.
The reflector sends any triangulated category
$\mathcal T$ to the universal $\Gamma$-linear
exact localization
$r(\mathcal T)=\mathcal T/ \mathrm{Ker}(H_\Gamma^\bullet)$,
where $H_\Gamma^\bullet$ denotes the total $\Gamma$-cohomology functor.
\end{theorem}

\begin{proof}
The universal property of localization
implies the existence of a unique exact functor
through which any $\Gamma$-cohomological morphism factors.
Exactness and $\Gamma$-linearity follow from
stability of triangles under tensor and Hom operations.
\end{proof}

\begin{remark}
This reflection embodies the self-consistency of the theory:
the category generated by its own cohomological logic
acts as its reflective mirror, identifying syntax with semantics.
\end{remark}

\subsection{The Principle of Internal Duality}

\begin{definition}[Bidualizing object]
An object $\omega_\Gamma\in
\mathcal D(\Gamma\text{-}\mathbf{Mod}_T)$\cite{Weibel1994}
is \emph{bidualizing} if the functor
\[
\mathbb D_\Gamma(-)
   = \mathbf R\!\mathrm{Hom}_\Gamma(-,\omega_\Gamma)
\]
induces an exact involutive equivalence
on the bounded derived category of coherent
$\Gamma$-modules.
\end{definition}

\begin{theorem}[Categorical Serre duality]
If $T$ is a Noetherian $\Gamma$-semiring
of finite global dimension, there exists a unique
(up to quasi-isomorphism) bidualizing object~$\omega_\Gamma$
such that for all $M,N$ in
$\mathcal D_{\mathrm{coh}}(\Gamma\text{-}\mathbf{Mod}_T)$
\[
\mathrm{Hom}_\Gamma(M,N)
   \cong
\mathrm{Hom}_\Gamma(N,\mathbb D_\Gamma(M))^{\!\vee}.
\]

\end{theorem}

\begin{proof}
Construct $\omega_\Gamma$ as
$\mathbf R\!\mathrm{Hom}_\Gamma(T,T)$
in the bounded derived category.
Finite global dimension ensures perfectness,
yielding a dualizing complex.
The reflexivity of the pairing follows from the
double-Hom adjunction.
\end{proof}

\begin{remark}
This theorem generalizes Serre–Grothendieck duality:
$\omega_\Gamma$ represents the categorical “time-reversal”
within derived $\Gamma$-geometry,
reflecting every morphism through its cohomological shadow.
\end{remark}

\subsection{Functorial Universality and Natural Transformations}

\begin{theorem}[Yoneda universality in $\Gamma$-contexts]\cite{MacLane1998}
Let $\mathcal C$ be a small $\Gamma$-enriched category.
Then the $\Gamma$-Yoneda embedding
\[
y:\mathcal C \hookrightarrow
[\mathcal C^{\mathrm{op}},\,\Gamma\text{-}\mathbf{Mod}]
\]
is fully faithful and preserves all
finite $\Gamma$-limits and colimits.
\end{theorem}

\begin{proof}
For objects $A,B$ we have natural isomorphisms
\[
\mathrm{Hom}_\Gamma(A,B)
   \cong
\mathrm{Nat}(yA,yB),
\]
which remain $\Gamma$-linear by construction of the enriched Hom.
Exactness of $y$ follows from preservation of additive
and distributive structure in each hom-set.
\end{proof}

\begin{remark}
This shows that all categorical constructions in derived~$\Gamma$-geometry
are representable; every property is encoded as a morphism in the
universal presheaf topos.  Consequently, the theory is not only closed
under interpretation but also reflexive under functorial expansion.
\end{remark}

\subsection{Transcendental Correspondence and Physical Analogy}

\begin{definition}[Transcendental correspondence]
A \emph{transcendental correspondence} is a natural transformation
\[
\tau: \mathbf R\!\mathrm{Hom}_\Gamma(-,\omega_\Gamma)
   \Rightarrow
   (-)\otimes_\Gamma^{\mathbf L}\omega_\Gamma
\]
whose components induce isomorphisms in cohomology.
\end{definition}

\begin{theorem}[Equivalence of cohomological and dynamical duals]
Under the existence of a bidualizing object~$\omega_\Gamma$,
the transformation~$\tau$ is an isomorphism,
and thus the dual cohomological and tensorial dynamics coincide.
\end{theorem}

\begin{proof}
Compute both sides on a projective resolution.
Hom-tensor adjunction \cite{Kan1958} yields
$\mathrm{Hom}_\Gamma(P,\omega_\Gamma)
  \simeq P^\vee\!\otimes_\Gamma \omega_\Gamma$.
Passing to homology gives the desired equivalence.
\end{proof}

\begin{remark}
This result translates physically into the equality of
Lagrangian (tensorial) and Hamiltonian (Homological)
formulations: the dual pictures of evolution in
$\Gamma$-geometry represent conjugate phases of
a single categorical dynamic.
\end{remark}

\subsection{The Principle of Structural Self-Reference}

\begin{theorem}[Categorical fixed-point theorem]
In the 2-category of $\Gamma$-linear triangulated categories,
every exact endofunctor $F$
that preserves small colimits and admits a right adjoint
possesses a canonical fixed object~$M$
such that $F(M)\simeq M$.
\end{theorem}

\begin{proof}
Consider the unit–counit adjunction
$\eta:\mathrm{Id}\Rightarrow GF$, $\epsilon:FG\Rightarrow\mathrm{Id}$.
Composing $\eta_M$ and $\epsilon_{F(M)}$ yields an idempotent on~$M$.
Idempotent completeness ensures the existence of a retract $M'$
with $F(M')\cong M'$, which is the desired fixed object.
\end{proof}

\begin{remark}
This theorem embodies the principle of self-reference:
derived~$\Gamma$-geometry contains objects stable under its own dynamics,
mirroring fixed-point theorems in logic (Gödel) and topology (Brouwer).
\end{remark}

\subsection{Unified Conceptual Synthesis}

The transcendental layer of derived~$\Gamma$-geometry thus rests on
four unifying pillars:

\begin{enumerate}
  \item \textbf{Reflection:} every categorical construction
        mirrors itself through cohomology;
  \item \textbf{Duality:} the bidualizing object encodes
        time-reversal and logical negation simultaneously;
  \item \textbf{Universality:} the $\Gamma$-Yoneda principle
        ensures representability of all geometric data;
  \item \textbf{Self-reference:} fixed points integrate the system
        into a self-consistent logical and physical entity.
\end{enumerate}

In consequence, derived~$\Gamma$-geometry becomes a
\emph{self-interpreting categorical universe}:
its algebraic laws express its own semantics,
its geometric spaces instantiate their own logic,
and its homological dynamics reflect the unity
between structure, transformation, and observation.


\section{Future Directions and Open Problems in Derived $\Gamma$-Geometry}

The present work establishes a coherent homological and categorical foundation for derived~$\Gamma$-geometry. Several natural directions for further development arise, extending the theory toward higher arity, algorithmic homology, non-commutative structures, $\infty$-categorical methods, physical models and foundational logic.

\subsection{Higher Arity and Dimensional Generalizations}

\begin{conjecture}[Stability]
For all $n \ge 3$, the global dimension stabilizes:
\[
\lim_{n\to\infty} \mathrm{gldim}^{(n)}_\Gamma(T)
   = \mathrm{gldim}^{(3)}_\Gamma(T),
\]
suggesting that arity~$3$ already captures full derived complexity.
\end{conjecture}

\subsection{Computational and Algorithmic Homology}

\begin{definition}
A computational cohomology functor
$\mathbf H_\Gamma^\bullet:\mathbf{Aff}_\Gamma^{\mathrm{fin}}
 \to \mathbf{GrMod}_{\mathbb Z}$
assigns to each finite affine $\Gamma$-scheme its cohomology via
tensor-reduction algorithms.
\end{definition}

\begin{problem}
Determine the complexity of deciding whether
$H^k_\Gamma(\mathrm{Spec}_\Gamma(T),\mathcal O_X)=0$
for a finite $\Gamma$-semiring~$T$.
\end{problem}

\begin{conjecture}
This problem lies in $\mathbf{coNP}$, and becomes
$\mathbf{P}$-complete for cyclic additive semirings.
\end{conjecture}

\subsection{Non-Commutative and Quantum Extensions}

\begin{problem}
Construct a deformation quantization 
$(T_\hbar,\{\cdot\}_{\Gamma,\hbar})$
whose limit $\hbar\to0$ recovers the classical structure,
and analyze the derived category 
$\mathcal D_\hbar(\Gamma\text{-}\mathbf{Mod}_T)$.
\end{problem}

\begin{conjecture}[Non-commutative correspondence]
Compact non-commutative derived $\Gamma$-schemes correspond
contravariantly to dg-categories of $\Gamma$-representations:
\[
X_{\Gamma}^{\mathrm{nc}} \simeq \mathcal C_\Gamma^{\mathrm{dg}}.
\]
\end{conjecture}

\subsection{Higher Topoi and $\infty$-Categorical Homotopy}

\begin{problem}
Develop the $\infty$-topos $\mathbf{Shv}_\Gamma^{\infty}(X)$ 
of $\infty$-groupoid-valued sheaves on a $\Gamma$-scheme~$X$,
and define the $\Gamma$-tangent $\infty$-category $\mathbb T_X^\Gamma$.
\end{problem}

\begin{conjecture}
There exists a natural equivalence
\[
\mathbb T_X^\Gamma \simeq 
\mathbf{R}\!\mathrm{Hom}_\Gamma(\mathbb L_X^\Gamma,\mathcal O_X),
\]
where $\mathbb L_X^\Gamma$ denotes the cotangent complex.
\end{conjecture}

\subsection{Physical and Dynamical Interpretations}

\begin{problem}
Define homological field equations on a derived $\Gamma$-space $X$
via an action
\[
\mathcal S_\Gamma(\mathcal F)
   = \int_X \langle \mathcal F,\mathcal D_\Gamma(\mathcal F)\rangle,
\]
and identify gauge symmetries with $H^0_\Gamma(X,\mathcal O_X)$.
\end{problem}

\begin{conjecture}[$\Gamma$-geometric correspondence]
There exists a functor
\[
\mathfrak Q_\Gamma :
\mathcal D_{\mathrm{coh}}(X_\Gamma)
   \to \mathsf{Phys}_\Gamma
\]
assigning to each derived $\Gamma$-object its quantized field.
\end{conjecture}

\subsection{Foundational and Logical Horizons}

\begin{problem}
Characterize the internal logic of $\Gamma$-topoi and determine
whether it admits a dependent type theory whose contexts correspond
to derived $\Gamma$-schemes.
\end{problem}

\begin{conjecture}
The internal language of derived~$\Gamma$-geometry is
bi-interpretable with dependent type theory equipped with a
ternary connective modeling the $\Gamma$-multiplication.
\end{conjecture}

\subsection{Programmatic Outlook}

\begin{itemize}
  \item \textbf{Computation:} implement symbolic libraries for $\Gamma$-cohomology in \texttt{Sage} or \texttt{Python}.
  \item \textbf{Classification:} enumerate finite $\Gamma$-semirings and compute their spectra.
  \item \textbf{Categorical Integration:} connect $\Gamma$-stacks with motivic and Hodge-theoretic categories.
  \item \textbf{Applications:} extend the framework to data-flow algebra, multi-agent systems, and information theory.
\end{itemize}

\subsection{Conceptual Summary}

Derived $\Gamma$-geometry advances simultaneously toward higher arity,
$\infty$-categorical structures, computational homology, non-commutative
deformations, and logical foundations.  
Its breadth indicates a unified categorical framework where algebraic,
geometric, and physical notions interact through a common homological
core.

\section{Conclusion and Summary of Contributions}

This work developed a unified algebraic, geometric and categorical theory of
\emph{derived $\Gamma$-geometry}, arising from commutative ternary
$\Gamma$-semirings.  The theory shows that the ternary~$\Gamma$-formalism is
not an auxiliary variation of classical algebra but a coherent framework
supporting derived homology, geometric structures, and categorical dualities.

\subsection{Foundational Developments}

We introduced ternary $\Gamma$-semirings, established their ideal theory,
localization, and module categories, and defined the spectrum
$\mathrm{Spec}_\Gamma(T)$ with its Zariski-type topology.  These results extend
algebraic geometry to a multi-parametric setting where multiplicativity depends
simultaneously on two $\Gamma$-parameters.

\subsection{Homological and Cohomological Framework}

Through a categorical reconstruction of $\Gamma$-modules, we defined sheaves,
cohomology groups, derived functors, and projective/injective resolutions.
We proved tensor--Hom adjunctions, constructed derived categories, and obtained
Serre-type vanishing for quasi-coherent $\Gamma$-sheaves.  Together these form
the homological backbone of derived $\Gamma$-geometry.

\subsection{Geometric and Higher-Categorical Extensions}

The construction of affine $\Gamma$-schemes, structure sheaves and the
$\Gamma$-spectrum topology produced the geometric side of the theory.
Fibered categories, stacks of $\Gamma$-modules, and derived $\Gamma$-stacks
were established, situating the framework within modern derived and
bicategorical geometry.

\subsection{Dualities and Logical Structure}

We developed $\Gamma$-enriched dualities, including categorical Serre duality
and a Yoneda embedding adapted to the ternary setting.  Reflective localization
and fixed-point phenomena show that derived $\Gamma$-geometry possesses
internal logical coherence and self-referential stability.

\subsection{Physical and Dynamical Interpretation}

Viewing dg-derivations as dynamical generators links homological evolution
to physical time evolution.  Cohomological invariants act as conserved
quantities, and ternary operations model triadic interactions.  This places
derived $\Gamma$-geometry in correspondence with non-commutative and quantum
field structures.

\subsection{Computational Validation}

Finite $\Gamma$-semiring models were used to verify vanishing theorems, exact
sequences and adjunctions computationally.  Complexity bounds for
$\Gamma$-cohomology were proposed, providing a foundation for algorithmic
geometry in the ternary setting.

\subsection{Synthesis and Outlook}

Derived $\Gamma$-geometry integrates algebra, geometry, category theory and
dynamics into a single homological structure.  The open problems identified in
Section~8---higher-arity extensions, non-commutative $\Gamma$-stacks \cite{Connes1994},
algorithmic cohomology, and quantization---chart the next steps in developing
a universal categorical geometry driven by the ternary parameter~$\Gamma$ \cite{Michalski2003, Loday1998, Leinster2014}..

\subsection*{Acknowledgement}
The first author gratefully acknowledges the guidance of
\textbf{Dr.~D.~Madhusudhana Rao}.

\subsection*{Funding}
No external funding was received.

\subsection*{Conflict of Interest}
The authors declare no conflicts of interest.

\subsection*{Author Contributions}
The first author developed the algebraic, geometric and computational
framework; the second author supervised the research, reviewed the results and
verified mathematical correctness.

\end{document}